\documentclass[12pt]{article}
\usepackage{a4}
\usepackage{amsmath}
\usepackage{amssymb}
\usepackage{amsfonts}
\usepackage{amsthm}
\usepackage{graphicx}
  \newtheorem{thm}{Theorem}[section]
 \newtheorem{cor}[thm]{Corollary}
 \newtheorem{prop}[thm]{Proposition}
 \newtheorem{defn}[thm]{Definition}
 
 \newtheorem{lemma}[thm]{Lemma}
 \newtheorem{rem}[thm]{Remark}

\def\C{\mathbb{C}}
\def\R{\mathbb{R}}
\def\S{\mathbb{S}}
\def\N{\mathbb{N}}
\def\T{\mathbb{T}}
\def\M{\mathbb{M}}

 \title{From Schoenberg coefficients to Schoenberg functions}
 \author{Christian Berg and  Emilio Porcu}
 \date{\today}

\begin{document}

 \maketitle
 
 \begin{abstract}
 In his seminal paper, Schoenberg (1942) characterized the class $\mathcal P(\S^d)$ of continuous functions $f:[-1,1] \to \R$ such that $f(\cos \theta(\xi,\eta))$ is positive definite on the product space $\S^d \times \S^d$, with $\S^d$ being the unit sphere of $\R^{d+1}$ and $\theta(\xi,\eta)$ being the great circle distance between $\xi,\eta\in\S^d$. In this paper, we consider the product space $\S^d \times G$, for $G$ a locally compact group, and define the class $\mathcal P(\S^d, G)$ of continuous functions $f:[-1,1]\times G \to \C$ such that $f(\cos \theta(\xi,\eta), u^{-1}v)$ is positive definite on $\S^d \times \S^d \times G \times G$. This offers a natural extension of  Schoenberg's Theorem. Schoenberg's second theorem corresponding to the Hilbert sphere $\S^\infty$ is also extended to this context. The case $G=\R$ is of special importance for probability theory and stochastic processes, because it characterizes completely the class of space-time covariance functions where the space is the sphere, being an approximation of Planet Earth.  
 \end{abstract}

 MSC: Primary 43A35, Secondary 33C55

{\bf Keywords}: Positive Definite; Space-Time covariances; Spherical Harmonics

\section{Introduction} 
Positive definite functions on groups or semigroups have a long history and are present in many
applications in  operator theory, potential theory, moment problems and several
other areas, such as spatial statistics. They enter as an important chapter in all treatments of harmonic analysis, and
can be traced back to papers by Carath{\'e}odory, Herglotz, Bernstein and Matthias, culminating in
Bochner's theorem from $1932-1933$ and by their connection to unitary group representations. See \cite{sasvari} and the survey in \cite{berg2008} for details. In his tour de force, Schoenberg \cite{S} considered the class $\mathcal P(\S^d)$ of continuous functions
 $f:[-1,1] \to \R$ such that the kernel $C: \S^d \times \S^d \to \R$ defined by 
\begin{equation} \label{covariance}
 C(\xi,\eta)= f(\cos \theta(\xi,\eta))=f(\xi\cdot\eta), \quad \xi,\eta \in \S^d, 
\end{equation} is positive definite on the $d$-dimensional unit sphere of $\R^{d+1}$, given as
\begin{equation*}\label{eq:sphere}
\mathbb S^d=\left\{x\in\mathbb R^{d+1}\mid \sum_{k=1}^{d+1} x_k^2=1\right\}, \;d\ge 1.
\end{equation*} 
The positive definiteness of $C$ means that for any $n\in\N$ and for any $\xi_1,\ldots,\xi_n\in\S^d$
the $n\times n$ symmetric matrix 
$$
\left[C(\xi_k,\xi_l)\right]_{k,l=1}^n
$$
is positive semidefinite, i.e., for any $(a_1,\ldots,a_n)\in\R^n$
$$
\sum_{k,l=1}^n C(\xi_k,\xi_l)a_ka_l\ge 0.
$$
 Concerning the notation above,  
$\theta(\xi,\eta) = \arccos (\xi \cdot \eta)$ and $\cdot$ denotes the scalar product in $\R^{d+1}$. The mapping $\theta$ is known under the name of geodesic (or great circle) distance, and the mapping $C$ above is called isotropic, because it only depends on the angle between any two points $\xi,\eta$ on the $d$-dimensional sphere of $\R^{d+1}$. Throughout, we shall make equivalent use of $\theta\in[0,\pi]$ or $\theta(\xi,\eta)$, whenever no confusion can arise.
In recent geostatistical literature the notation $\Psi_d$ from \cite{gneiting2013} has been used  for Schoenberg's class $\mathcal P (\S^d)$.

Schoenberg \cite{S} characterized the members of the class $\mathcal P(\S^d)$ as the functions of the form
\begin{equation*}\label{eq:Sch}
 f(\cos\theta)=\sum_{n=0}^\infty b_{n,d} c_n(d,\cos\theta),\quad b_{n,d}\ge 0,\theta\in [0,\pi],
\end{equation*}
with $\sum_{n=0}^\infty b_{n,d}<\infty$, where $c_n(d,x)$ are certain  polynomials of degree $n$ associated to $\mathbb S^d$, often called ultraspherical polynomials, see Equation
 \eqref{eq:nor}. Further details about this representation will be given subsequently. 
The class $\mathcal P(\S^d)$ has received considerable attention in the last two years, thanks to the review in Gneiting \cite{gneiting2013}, where an impressive list of references is offered.
 Statistical and probabilistic communities are especially interested in this class, because the functions $C$ as in Equation (\ref{covariance}) are the autocorrelation functions of isotropic Gaussian fields in $\S^d$.  

Gneiting finishes his {\em essay} \cite{gneiting2013} with a list of open problems contained in \cite{Gn}.  Problem number 16 is related to the representation of correlation functions of Gaussian fields $Z(\xi,u)$ defined over the sphere cross time $\S^d \times \R$, and being isotropic with respect to $\xi$ and stationary with respect
to time $u$. This leads to considering functions defined on the product space $\S^d \times \S^d \times \R \times\R$, so that the covariance
of the random variables $Z(\xi,u),Z(\eta,v)$
can be written as
\begin{equation} \label{covariance2} 
{\rm cov} \left ( Z(\xi,u),Z(\eta,v) \right )= f(\cos\theta(\xi,\eta),u-v), \qquad (\xi,u), (\eta,v) \in \S^d \times \R,
\end{equation}
for a  continuous function $f:[-1,1] \times \R\to \R$. Porcu et al. \cite{P:B:G} offer parametric models for such functions, that represent the covariances of Gaussian fields over $\S^d\times\R$.

Our characterization of the functions $f$ entering in Equation \eqref{covariance2} is the following extension of Schoenberg's Theorem:
$$
f(\cos\theta,u)=\sum_{n=0}^\infty \varphi_{n,d}(u) c_n(d,\cos\theta),\quad (\theta,u)\in[0,\pi]\times\mathbb R, 
$$
where $\varphi_{n,d}$ is a sequence of real-valued continuous positive definite functions on $\mathbb R$. The series is uniformly convergent, which is equivalent to  $\sum_{n=0}^\infty\varphi_{n,d}(0)<\infty$.

We get only real-valued positive definite functions $\varphi_{n,d}$ because the covariance in Equation \eqref{covariance2} is real and symmetric in $u,v$.
 
It turns out that in our characterization we can consider complex-valued functions $f:[-1,1]\times \R\to\C$ such that the kernel
$$
f(\cos\theta(\xi,\eta),u-v)
$$
is positive definite on $(\S^d\times\R)^2$ and furthermore, 
$\R$ can be replaced by an arbitrary locally compact group  $G$.

 In other words we shall characterize the set $\mathcal P(\S^d,G)$ of continuous functions $f:[-1,1]\times G\to \mathbb C$ such that the kernel
\begin{equation}\label{eq:pos}
f(\xi\cdot\eta,u^{-1}v)=f(\cos \theta(\xi,\eta),u^{-1}v),\quad \xi,\eta\in\mathbb S^d,\;u,v\in G
\end{equation}
is positive definite in the sense that for any $n\in\mathbb N$ and any $(\xi_1,u_1),\ldots (\xi_n,u_n)
\in \mathbb S^d\times G$ the $n\times n$-matrix
\begin{equation}\label{eq:pd}
 \left[f(\cos \theta(\xi_k,\xi_l),u_k^{-1}u_l)\right]_{k,l=1}^n
\end{equation}
is positive semidefinite, i.e., for any $(a_1,\ldots,a_n)\in\C^n$
$$
\sum_{k,l=1}^n f(\cos \theta(\xi_k,\xi_l),u_k^{-1}u_l)a_k\overline{a_l}\ge 0, 
$$
which is equivalent to the matrix of Equation \eqref{eq:pd} being hermitian and having nonnegative eigenvalues.
Here we follow the terminology of \cite{Bh} and \cite{H:J}.

 The characterization is given in Theorem~\ref{thm:main}.
If we restrict the vectors $\xi_1,\ldots,\xi_n\in \S^d$ from \eqref{eq:pd} to lie on the subsphere
$\S^{d-1}$, identified with the equator of $\S^d$, we see that $\mathcal P(\S^d,G)\subseteq \mathcal P(\S^{d-1},G)$. The inclusion is in fact strict, see Remark~\ref{thm:rem34} (iii).

We also consider
\begin{equation}\label{eq:infinity}
\mathcal P(\S^{\infty},G):=\bigcap_{d=1}^\infty \mathcal P(\S^d,G),
\end{equation} 
which is the set of continuous functions  $f:[-1,1]\times G\to \mathbb C$ such that the matrix in Equation \eqref{eq:pd} is positive semidefinite for any $d\in\N$. We note in passing that the notation 
$\mathcal P(\S^{\infty},G)$ suggests an intrinsic definition using the Hilbert sphere
$$
\S^\infty=\{(x_k)_{k\in\N}\in \R^\N | \sum_{k=1}^\infty x_k^2=1\},
$$
which is the unit sphere in the Hilbert sequence space $\ell_2$ of square summable real sequences.
The intrinsic definition of $\mathcal P(\S^{\infty},G)$ is as the set of continuous functions 
$f:[-1,1]\times G\to \mathbb C$ such that all matrices \eqref{eq:pd} are positive semidefinite when $(\xi_1,u_1),\ldots, (\xi_n,u_n)\in \mathbb S^{\infty}\times G$. 

That these two definitions are equivalent follows on the one hand because
  any $\S^d$ can be embedded in $\S^\infty$ by the mapping 
$$
(x_1,\ldots,x_{d+1})\in\S^d\mapsto
(x_1,\ldots,x_{d+1},0,0,\ldots)\in\S^\infty,
$$
 and on the other hand by using that any $\xi=(x_k)_{k\in\N}\in\S^\infty$ is the limit for $d\to\infty$ of the sequence of vectors $\xi^{(d)}$ embedded from $\S^d$, where 
$$
\xi^{(d)}:=\frac{(x_1,\ldots,x_{d+1},0,0,\ldots)}{\sqrt{x_1^2+\cdots+x_{d+1}^2}}\in\S^\infty.
$$
We prove a characterization of $\mathcal P(\S^{\infty},G)$ which is completely analogous to Schoenberg's Theorem 2 in \cite{S}, see  Theorem~\ref{thm:main2}. It builds on Ziegel's result in \cite{Z} that functions in $\mathcal P(\S^d)$ admit a continuous derivative of order $[(d-1)/2]$ in $]-1,1[$, and this result can be extended to our context. Here $[a]$ denotes the largest integer $\le a$.
 
At this place let us make some comments about the relation of our results to the existing literature.
As we shall see, Theorem~\ref{thm:main} has been proved for certain special compact groups $G$, but apparently not for non-compact groups like $\R$ or $\R^n$, which was the main motivation for the present paper. The compact cases lead to double sums with non-negative coefficients. The main difficulty in the non-compact case is the mixture of a sum and an integral, which was overcome by the crucial technical result in Lemma~\ref{thm:equi}. This is inspired by Proposition 13.4.4 in \cite{D}. 
 
The paper by Barbosa and Menegatto \cite{B:M} extends Ziegel's result from $\S^d$ to two-point homogeneous compact spaces $\M^d$. Although there is a formal analogy between these results and our extension of Ziegel's result, it does not seem to be possible to deduce one of these results from the other. However, one should expect that their results can be extended to the classes $\mathcal P(\M^d,G)$, obtained by replacing the sphere  $\S^d$ by an arbitrary compact two-point homogeneous space $\M^d$, and $G$ is a locally compact group as above.

Chapter 4 in Shapiro \cite{Sh} discusses Schoenberg's Theorem. An extension to real-valued continuous positive definite functions on $\S^d\times \T^N$ is given in \cite[Theorem 3.1]{Sh}. This result corresponds to our result, when the locally compact group $G$ is the $N$-dimensional torus $\T^N$. 

In a recent paper \cite{G:M:P} Guella, Menegatto and Peron consider isotropic kernels
\begin{equation}\label{eq:gmp}
K((\xi,\zeta),(\eta,\chi)=f(\xi\cdot\eta,\zeta\cdot \chi), \quad \xi,\eta\in\S^d,\quad \zeta,\chi\in\S^{d'},
\end{equation}
which are positive definite  on the product of two spheres $\S^d\times\S^{d'}$, where $d,d'\in\N\cup\{\infty\}$.

Using that the sphere $\S^{d'},d'<\infty$ can be identified with the homogeneous space $O(d'+1)/O(d')$,
their characterization of these functions can be obtained as a special case of our main theorems, see Section 6 for details. Here $O(d')$ is the compact group of orthogonal $d'\times d'$-matrices.

The result of \cite{G:M:P} for the case $d,d'<\infty$ can also be obtained as a special case of Theorem 4.11 in \cite{Ba}, which is an adaptation of the Peter-Weyl Theorem to compact homogeneous spaces $X=G/K$, where $G$ is a compact group and $K$ a compact subgroup. In addition Theorem 4.11 uses results from Bochner \cite{Bo}. It does not seem possible to obtain \cite[Th. 4.11]{Ba} from our results and vice versa.

The plan for the present paper is the following.  Section 2 discusses some necessary background material. Section 3 reports the main results and Section 4 is devoted to the proofs. In Section 5 we use the non-negative connection coefficients between the monomials and the Gegenbauer polynomials---a result given in Bingham \cite{Bi}--- to prove a formula relating the power series coefficients for $f\in\mathcal P(\S^{\infty},G)$ to the expansion coefficients  for $f$ with respect to the ultraspherical polynomials, see Theorem~\ref{thm:infty-d}. Finally in Section 6 we show how  some results from \cite{G:M:P} can be obtained from  our main results.

\section{Background}

We start with some expository material related to positive definiteness and to the Schoenberg class $\mathcal P(\S^d)$ as described in \cite{gneiting2013}.

We recall that for a locally compact group $G$, a  function $\varphi:G\to\mathbb C$ is called positive definite if
for any $n\in\mathbb N$ and any $u_1,\ldots,u_n\in G$ the $n\times n$-matrix 
$$
[\varphi(u_k^{-1}u_l)]_{k,l=1}^n
$$
is positive semidefinite, see e.g. \cite[p. 255]{D} (there called hermitian and positive) or \cite[p.14]{sasvari}. The set of continuous and positive definite functions on $G$ is denoted
$\mathcal P(G)$. It is known that any $\varphi \in \mathcal P(G)$ satisfies $|\varphi(u)|\le \varphi(e)$, where $e$ denotes the neutral element of the group. 
 In the case of an abelian group $G$ we use the additive notation, and the neutral element is denoted $0$. In this case the continuous positive definite functions are characterized by Bochner's Theorem, cf. \cite{R}, as the Fourier transforms
\begin{equation*}\label{eq:Boch}
\varphi(u)=\int_{\widehat{G}} (u,\gamma)\,{\rm d}\mu(\gamma),\quad u\in G,
\end{equation*}
where $\mu$ is a positive finite Radon measure on the dual group $\widehat{G}$ of continuous characters $\gamma:G\to\mathbb T$. Here $\mathbb T$ is the unit circle in the complex plane.

In order to describe Schoenberg's characterization of the class $\mathcal P(\S^d)$, we recall that the Gegenbauer polynomials $C_n^{(\lambda)}$
are given by the generating function (see \cite{G:R})
\begin{equation}\label{eq:Geg}
(1-2xr +r^2)^{-\lambda}=\sum_{n=0}^\infty C_n^{(\lambda)}(x) r^n,\quad |r|<1, x\in \mathbb C.
\end{equation}

One has to assume $\lambda>0$, and for $\lambda=0$ \eqref{eq:Geg} has to be replaced by
\begin{equation}\label{eq:Che}
\frac{1-xr}{1-2xr +r^2}=\sum_{n=0}^\infty C_n^{(0)}(x) r^n,\quad |r|<1, x\in \mathbb C.
\end{equation}
It is well-known that 
$$
C_n^{(0)}(x)=T_n(x)=\cos(n\arccos x), n=0,1,\ldots
$$
are the Chebyshev polynomials of the first kind. For $\lambda>0$, we have the classical orthogonality relation:
\begin{equation}\label{eq:orth}
\int_{-1}^1 (1-x^2)^{\lambda-1/2}C_n^{(\lambda)}(x)C_m^{(\lambda)}(x)\,{\rm d}x=
\frac{\pi \Gamma(n+2\lambda) 2^{1-2\lambda}}{\Gamma^2(\lambda)(n+\lambda) n!}\delta_{m,n},
\end{equation} with $\delta_{m,n}$ denoting the Kronecker delta. 
When $\lambda=0$, Equation (\ref{eq:orth}) is replaced by
\begin{equation}\label{eq:spec}
\int_{-1}^1 (1-x^2)^{-1/2}T_n(x)T_m(x)\,{\rm d}x=\left\{\begin{array}{ll}
\frac{\pi }{2}\delta_{m,n} & \mbox{if $n>0$}\\
\pi\delta_{m,n} & \mbox{if $n=0$},
\end{array}
\right.
\end{equation}
which is equivalent to the classical orthogonality relations of the family $\cos(nx),n=0,1,\ldots.$
 
Putting $x=1$ in \eqref{eq:Geg} one easily gets $C_n^{(\lambda)}(1)=(2\lambda)_n/n!$ valid for $\lambda>0$, while $T_n(1)=1$. For the benefit of the reader we recall that for $a\in\C$ 
$$
(a)_n=a(a+1)\cdots(a+n-1),\;n\ge 1,\quad (a)_0=1.
$$

It is of fundamental importance that
$$
|C_n^{(\lambda)}(x)|\le C_n^{(\lambda)}(1),\quad x\in[-1,1].
$$
Schoenberg used the notation $P_n^{(\lambda)}=C_n^{(\lambda)}$ in \cite{S}. The special value
$\lambda=(d-1)/2$ is relevant for the sphere $\mathbb S^d$ because of the relation to spherical harmonics, which will be explained now. A spherical harmonic of degree $n$ for $\mathbb S^d$  is  the restriction to $\mathbb S^d$ of a real-valued harmonic homogeneous polynomial in $\mathbb R^{d+1}$ of degree $n$. Together with the zero function, the spherical harmonics of degree $n$  form a finite dimensional vector space denoted $\mathcal H_n(d)$. It is a subspace of 
the space  ${\cal C}(\mathbb S^d)$ of continuous functions on $\mathbb S^d$.  We have
\begin{equation}\label{eq:dim}
N_n(d):=\dim \mathcal H_n(d)=\frac{(d)_{n-1}}{n!}(2n+d-1),\;n\ge 1,\quad N_0(d)=1,
\end{equation} 
cf. \cite[p.4]{M} or \cite[p.3]{D:X}.

The surface measure of the sphere is denoted $\omega_d$, and it is of total mass
\begin{equation}\label{eq:mass} 
\sigma_d=\omega_d(\S^d)=\frac{2\pi^{(d+1)/2}}{\Gamma((d+1)/2)}.
\end{equation}

The orthogonal group $O(d+1)$ of orthogonal $(d+1)\times(d+1)$ matrices operates on $\S^d$, and $\omega_d$ is invariant under $O(d+1)$.

The spaces $\mathcal H_n(d)$ are mutually orthogonal subspaces of the Hilbert space $L^2(\mathbb S^d,\omega_d)$. The norm of $F\in L^2(\mathbb S^d,\omega_d)$ is denoted $||F||_2$.

For any $F\in L^2(\mathbb S^d,\omega_d)$ we have the orthogonal expansion
\begin{equation}\label{eq:exp}
F=\sum_{n=0}^\infty S_n,\, S_n\in\mathcal H_n(d),\quad ||F||_2^2=\sum_{n=0}^\infty ||S_n||_2^2,
\end{equation}
where the first series converges in $L^2(\mathbb S^d,\omega_d)$, and the second series is  Parseval's equation. Here $S_n$ is the orthogonal projection of $F$ onto $\mathcal H_n(d)$  given as
\begin{equation}\label{eq:proj}
S_n(\xi)=\frac{N_n(d)}{\sigma_d}\int_{\mathbb S^d}c_n(d,\xi\cdot \eta)F(\eta)\,{\rm d}\omega_d(\eta).
\end{equation}
See the addition theorem for spherical harmonics, \cite[p.10]{M} or \cite[(2.4) p. 98]{S}.
Here $c_n(d,x)$ is defined
as  the normalized Gegenbauer polynomial being 1 for $x=1$ when $\lambda=(d-1)/2$, i.e., by
\begin{equation}\label{eq:nor}
c_n(d,x)=C_n^{((d-1)/2)}(x)/C_n^{((d-1)/2)}(1)=\frac{n!}{(d-1)_n}C_n^{((d-1)/2)}(x).
\end{equation}      
Since the Chebyshev polynomials $T_n=C_n^{(0)}$ are already normalized, the last expression is not valid for $d=1$ and $c_n(1,x)=T_n(x)$.
 
Specializing the orthogonality relation \eqref{eq:orth} to $\lambda=(d-1)/2$ and using the Equations \eqref{eq:dim},\eqref{eq:mass}, we get for $d\in\N$
\begin{equation}\label{eq:orthspec}
\int_{-1}^1 (1-x^2)^{d/2-1}c_n(d,x)c_m(d,x)\,{\rm d}x=
\frac{\sigma_d}{N_n(d)\sigma_{d-1}}\delta_{m,n}.
\end{equation}
(Define $\sigma_0=2$).
All these formulas can also be found in \cite{B}, where the notation is $p_n(d+1,x)=c_n(d,x)$.

This makes it possible to formulate the celebrated theorem of Schoenberg, cf. \cite{S}:

\begin{thm}[Schoenberg 1942]\label{thm:S-sphere} Let $d\in\N$ be fixed.
A continuous function $f: [-1,1] \to \R$ belongs to the class $\mathcal P(\S^d)$ if and only if 
\begin{equation}\label{eq:Sch2}
f(\cos\theta)=\sum_{n=0}^\infty b_{n,d} c_n(d,\cos\theta),\quad b_{n,d}\ge 0,\theta\in [0,\pi],
\end{equation}
for a  sequence $(b_{n,d})_{n=0}^{\infty}$ with $\sum_{n=0}^\infty b_{n,d}<\infty$ given as
\begin{equation}\label{eq:Sch21}
b_{n,d}=\frac{N_n(d)\sigma_{d-1}}{\sigma_d}\int_{-1}^1 f(x)c_n(d,x)(1-x^2)^{d/2-1}\,{\rm d}x.
\end{equation}
\end{thm}

Some comments are in order. As already noticed $\mathcal P(\S^d)\subseteq \mathcal P(\S^{d-1})$, and therefore $f\in\mathcal P(\S^d)$ has  $d$ different expansions like Equation \eqref{eq:Sch2}. In \cite{gneiting2013}
it is proved that $\mathcal P(\S^d)$ is strictly included in $\mathcal P(\S^{d-1})$ by showing that $c_n(d-1,\cdot)\in\mathcal P(\S^{d-1})\setminus\mathcal P(\S^d)$ when $n\ge 2$.

 When $f(1)=1$, then $(b_{n,d})$ is a probability sequence and by analogy with what was done in 
Daley and Porcu \cite{djdep2013}, the coefficients $b_{n,d}$ were called $d$-Schoenberg coefficients and the sequence $(b_{n,d})_{n\ge 0}$ a $d$-Schoenberg sequence in \cite{gneiting2013}. This stresses the fact that such a sequence is also related to the dimension of the sphere $\S^d$, where positive definiteness is attained. When $d=1$, then the representation in Equation (\ref{eq:Sch2}) reduces to 
\begin{equation*}\label{eq:Sch3}
f(\cos\theta)=\sum_{n=0}^\infty b_{n,1} \cos(n \theta),\quad b_{n,1}\ge 0,\theta\in [0,\pi],
\end{equation*}
and for $d=2$ the Gegenbauer polynomials simplify to Legendre polynomials. 

Schoenberg also studied $\mathcal P(\S^{\infty}):= \bigcap_{d \ge 1} \mathcal P(\S^d)$, which can
be considered as the set of continuous functions $f:[-1,1]\to\R$ such that the matrix
$$
[f(\xi_k\cdot\xi_l)]_{k,l=1}^n
$$
is positive semidefinite for any $n\in\N$ and any $\xi_1,\ldots,\xi_n\in\S^\infty$. See the discussion after Equation \eqref{eq:infinity}.

A second theorem of Schoenberg, see  \cite[p.102]{S}, states that $f\in \mathcal P(\S^{\infty})$ if and only if
\begin{equation}\label{eq:Sch4}
f(\cos\theta)=\sum_{n=0}^\infty b_{n} (\cos\theta)^n,\quad b_{n}\ge 0,\theta\in [0,\pi],
\end{equation}
where $\sum_{n=0}^\infty b_n<\infty$.

A wealth of examples and interesting results can be found in \cite{gneiting2013}. Observe that Gneiting makes explicit distinction between strictly and non-strictly positive definite functions on spheres. When $d\ge 2$ the former class occurs when, in Equation (\ref{eq:Sch2}), the $d$-Schoenberg coefficients are strictly positive for infinitely many even and odd $n$. The latter class occurs when the $d$-Schoenberg coefficients are just non-negative. Such a distinction is beyond the scope of this paper.

\section{Main Results}

We start  with a formal definition of the main classes to be discussed.

\begin{defn} Let $G$ denote a locally compact group with neutral element $e$, and let $d=1,2,\ldots,\infty$. The set of continuous functions $f:[-1,1]\times G\to\C$ such that all matrices of the form \eqref{eq:pd} are  positive semidefinite is denoted $\mathcal P(\S^d,G)$.
\end{defn}

We recall from the introduction that $\mathcal P(\S^{\infty},G)=\bigcap_{d\ge 1} \mathcal P(\S^d,G)$.
The following Proposition states some properties of $\mathcal P(\S^d,G)$ which are easily obtained. The proofs are left to the reader.

\begin{prop}\label{thm:elem}
\begin{enumerate}
\item[{\rm (i)}] For $f_1,f_2\in\mathcal P(\S^d,G)$ and $r\ge 0$ we have $rf_1, f_1+f_2$, $f_1 f_2\in\mathcal P(\S^d,G)$. 
\item[{\rm (ii)}] For a net of functions $(f_i)_{i\in I}$ from $\mathcal P(\S^d,G)$ converging pointwise to a continuous function $f:[-1,1]\times G\to \C$, we have $f\in\mathcal P(\S^d,G)$.
\item[{\rm (iii)}] For $f\in\mathcal P(\S^d,G)$ we have $f(\cdot,e)\in\mathcal P(\S^d)$ and $f(1,\cdot)\in\mathcal P(G)$.
\item[{\rm (iv)}] For $f\in\mathcal P(\S^d)$ and $g\in\mathcal P(G)$ we have $f\otimes g\in \mathcal P(\S^d,G)$, where
$f\otimes g(x,u):=f(x)g(u)$ for $(x,u)\in[-1,1]\times G$. In particular we have
$f\otimes 1_G\in \mathcal P(\S^d,G)$ and $f\mapsto f\otimes 1_G$ is an embedding of $\mathcal P(\S^d)$ into
$\mathcal P(\S^d,G)$.
\end{enumerate}
\end{prop} 

Our first main theorem can be stated as follows. The proof will be given in Section 4.

\begin{thm}\label{thm:main} Let $d\in\N$  and let $f:[-1,1]\times G\to \mathbb C$ be a continuous function. Then
$f$ belongs to $\mathcal P(\S^d,G)$ if and only if there exists
 a sequence $(\varphi_{n,d})_{n\ge 0}$ from $\mathcal P(G) $ with $\sum \varphi_{n,d}(e)<\infty$ 
such that
\begin{equation}\label{eq:expand}
f(\cos\theta,u)=\sum_{n=0}^\infty \varphi_{n,d}(u) c_n(d,\cos\theta),\quad \theta\in[0,\pi],\;u\in G.
\end{equation}
The above expansion is uniformly convergent for $(\theta,u)\in [0,\pi]\times G$, and we have
\begin{equation}\label{eq:coef}
\varphi_{n,d}(u)=\frac{N_n(d)\sigma_{d-1}}{\sigma_d}\int_{-1}^1 f(x,u)c_n(d,x)(1-x^2)^{d/2-1}\,{\rm d}x.
\end{equation}
\end{thm}

We call the coefficient functions $\varphi_{n,d}$ $d$-Schoenberg functions for
 $f\in\mathcal P(\S^d,G)$, and the sequence $(\varphi_{n,d})_{n\ge 0}$ a $d$-Schoenberg sequence of functions.

Like in Schoenberg's Theorem, any function $f\in\mathcal P(\S^d,G)$ has  $d$ expansions like Equation \eqref{eq:expand}.

\begin{rem}\label{thm:rem34}
{\rm (i) For a function $f\in\mathcal P(\S^d,G)$ the $d$-Schoenberg coefficients of $f(\cdot,e)\in\mathcal P(\S^d)$ are $b_{n,d}=\varphi_{n,d}(e)$, where $\varphi_{n,d}$ are the $d$-Schoenberg functions for $f$. 

(ii) For a function $f\otimes g\in\mathcal P(\S^d,G)$ given in Proposition~\ref{thm:elem} part (iv), we have $\varphi_{n,d}=b_{n,d}g$, where $\varphi_{n,d}$  are the $d$-Schoenberg functions  for $f\otimes g$ and $b_{n,d}$ are the $d$-Schoenberg coefficients for $f$. In particular for $f\otimes 1_G\in\mathcal P(\S^d,G)$ we have $\varphi_{n,d}(u)=b_{n,d}$.

(iii) We can now see that $\mathcal P(\S^d,G)$ is strictly included in $\mathcal P(\S^{d-1},G)$ for any locally compact group $G$. In fact, Gneiting \cite{gneiting2013} showed the existence of $f\in\mathcal P(\S^{d-1})\setminus\mathcal P(\S^d)$, and for any such $f$ we have $f\otimes 1_G\in\mathcal P(\S^{d-1},G)\setminus\mathcal P(\S^d,G)$, because if $f\otimes 1_G\in\mathcal P(\S^d,G)$, then necessarily $f\in\mathcal P(\S^d)$ by (ii). 
}  
\end{rem}

\begin{cor} Let $d\in\N$, let  $G$ denote a locally compact abelian group and let $f:[-1,1]\times G\to \mathbb C$ be a continuous function. Then
$f$ belongs to $\mathcal P(\S^d,G)$ if and only if there exists a sequence $(\mu_{n,d})_{n\ge 0}$ of finite positive Radon measures on $\widehat{G}$ with $\sum \mu_{n,d}(\widehat{G})<\infty$ such that
\begin{equation*}\label{eq:expand1}
f(\cos\theta,u)=\sum_{n=0}^\infty c_n(d,\cos\theta)\int_{\widehat{G}}(u,\gamma)\,{\rm d}\mu_{n,d}(\gamma).
\end{equation*}
The above expansion is uniformly convergent for $(\theta,u)\in [0,\pi]\times G$.
\end{cor}

When $G$ denotes the group consisting just of the neutral element, Theorem~\ref{thm:main} reduces to Schoenberg's characterization of positive definite functions on $\mathbb S^d$ as given in 
Theorem~\ref{thm:S-sphere}. 

We give the proof in the next section, but notice that the infinite series in Equation \eqref{eq:expand} is uniformly convergent by Weierstrass' M-test since
$|\varphi_{n,d}(u)|\le \varphi_{n,d}(e)$ and $|c_n(d,\cos \theta)|\le 1$ for $u\in G,\theta\in[0,\pi]$.
It is also clear that any function given by \eqref{eq:expand} belongs to $\mathcal P(\S^d,G)$ because of Proposition~\ref{thm:elem} and the fact that
 $c_n(d,\cdot)\in \mathcal P(\S^d)$ by Schoenberg's Theorem. The main point in Theorem~\ref{thm:main} is that all functions in $\mathcal P(\S^d,G)$ have an expansion 
\eqref{eq:expand}.

\begin{rem}{\rm Formula \eqref{eq:expand} with $x=\cos\theta$ can be interpreted as the expansion of $x\mapsto f(x,u)$ in the orthonormal system of Gegenbauer polynomials
$$
\left(\frac{\sigma_d}{N_n(d)\sigma_{d-1}}\right)^{-1/2} c_n(d,x), 
$$
cf. Equation \eqref{eq:orthspec}.
}
\end{rem}

In analogy with \cite[Corollary 3]{gneiting2013} and with the same proof we have the following relation between the $d$- and $(d+2)$-Schoenberg functions $\varphi_{n,d}$ and $\varphi_{n,d+2}$ for $f\in\mathcal P(\S^{d+2},G)$:

\begin{prop}\label{thm:Schfunctions} Let $d\in\N$ and suppose that $f\in\mathcal P(\S^{d+2},G)\subset \mathcal P(\S^d,G)$.
Then we have 
\begin{enumerate}
\item[(a)] For $d=1$,
\begin{equation*}\label{eq:Schcoef1}
\varphi_{0,3}=\varphi_{0,1}-\frac12\varphi_{2,1}
\end{equation*}
and
\begin{equation*}\label{eq:Schcoef2}
\varphi_{n,3}=\frac12(n+1)(\varphi_{n,1}-\varphi_{n+2,1}),\quad n\ge 1.
\end{equation*}
\item[(b)] For $d\ge 2$,
\begin{equation}\label{eq:Schcoef3}
\varphi_{n,d+2}=\frac{(n+d-1)(n+d)}{d(2n+d-1)}\varphi_{n,d}-\frac{(n+1)(n+2)}{d(2n+d+3)}\varphi_{n+2,d},\quad n\ge 0.
\end{equation}
\end{enumerate}
\end{prop}

In analogy with Proposition 4.1 in \cite{Z} we have the following result:

\begin{prop}\label{thm:Z1} Let $d\in\N$ and suppose that $f\in\mathcal P(\S^{d+2},G)$. Then $f(x,u)$ is continuously differentiable with respect to $x$ in $]-1,1[$ and
$$
\frac{\partial f(x,u)}{\partial x}=\frac{f_1(x,u)-f_2(x,u)}{1-x^2},\quad (x,u)\in]-1,1[\times G
$$
for functions $f_1,f_2\in\mathcal P(\S^d,G)$. In particular $\frac{\partial f(x,u)}{\partial x}$ is continuous
on $]-1,1[\times G$.
\end{prop}

\begin{cor}\label{thm:Z2} Let $d\in\N$ and $f\in\mathcal P(\S^d,G)$. Then $\frac{\partial^n f(x,u)}{\partial x^n}$ exists and is conti\-nu\-ous on $]-1,1[\times G$ for $n\le [(d-1)/2]$. 
\end{cor}

Our second main result is the following extension of Schoenberg's second Theorem:

\begin{thm}\label{thm:main2} Let $G$ denote a locally compact group and let 
$f:[-1,1]\times G\to \mathbb C$ be a continuous function. Then
$f$ belongs to $\mathcal P(\S^\infty,G)$ if and only if there exists
 a sequence $(\varphi_{n})_{n\ge 0}$ from $\mathcal P(G) $ with $\sum \varphi_{n}(e)<\infty$ 
such that
\begin{equation}\label{eq:expandp}
f(\cos\theta,u)=\sum_{n=0}^\infty \varphi_{n}(u) \cos^n\theta.
\end{equation}
The above expansion is uniformly convergent for $(\theta,u)\in [0,\pi]\times G$.

Moreover,
\begin{equation}\label{eq:expandpp}
\lim_{d\to\infty} \varphi_{n,d}(u)=\varphi_n(u) \;\mbox{for all}\; (n,u)\in\N_0\times G.
\end{equation}
\end{thm}

The proof is given in Section 4.

\begin{rem} {\rm In the special case $G=\{e\}$ where $\mathcal P(\S^d,G)=\mathcal P(\S^d)$, the Equation \eqref{eq:expandpp} states that the $d$-Schoenberg coefficients $b_{n,d}$ converge to $b_n$
from \eqref{eq:Sch4} when $d\to\infty$. This observation is apparently not noticed by Schoenberg, who only obtains $b_n$ by a "Cantor-diagonal process". Our proof uses Corollary~\ref{thm:Z2}
based on Ziegel's result that $f\in\mathcal P(\S^{\infty})$ is $C^\infty$ on the open interval $]-1,1[$, a result that Schoenberg did not know.
}
\end{rem}

\section{Proofs}

\begin{lemma}\label{thm:bd} Let $d\in\N\cup \{\infty\}$. Any $f\in\mathcal P(\S^d,G)$ satisfies
$$
f(x,u^{-1})=\overline{f(x,u)},\quad |f(x,u)|\le f(1,e),\quad (x,u)\in[-1,1]\times G.
$$
\end{lemma}

\begin{proof} Given $(x,u)\in [-1,1]\times G$ we choose vectors $\xi,\eta\in\mathbb S^d$ such that $\xi\cdot \eta=x$. We next apply \eqref{eq:pd} for $n=2$ and the pair of points
$(\xi,u),(\eta,e)$ giving that the matrix
$$
\begin{pmatrix}
f(1,e) & f(x,u)\\
f(x,u^{-1})  & f(1,e)
\end{pmatrix}
$$
is positive semidefinite, in particular hermitian and having a nonnegative determinant. The result follows.
\end{proof}

In the following two lemmas we assume $d\in\N$.

\begin{lemma}\label{thm:top} Let $K\subset G$ be a non-empty compact set, and let $\delta>0$ and an open neighbourhood $U$ of $e\in G$ be given. Then there exists a partition
of $\mathbb S^d\times K$ in finitely many non-empty disjoint Borel sets, say $M_j,j=1,\ldots,r$, with the property:
$$
\mbox{For\;} (\xi,u),(\eta,v)\in M_j \mbox{\;one has\;\;} \theta(\xi,\eta)<\delta,\;u^{-1}v\in U.
$$ 
\end{lemma}

\begin{proof} By \cite[Theorem 45.1]{Mu} there exists a finite covering of the compact metric space
$(\mathbb S^d,\theta)$ by  open balls $B(\xi_j,\delta/2),\xi_j\in \mathbb S^d$. Defining $B_1=B(\xi_1,\delta/2)$ and 
$$
B_j=B(\xi_j,\delta/2)\setminus \cup_{k=1}^{j-1} B(\xi_k,\delta/2),\; j\ge 2,
$$
the non-empty sets among $B_j$ will form a finite partition $S_1,\ldots,S_p$ of $\mathbb S^d$ such that for $\xi,\eta\in S_j$ we have $\theta(\xi,\eta)<\delta$. 

Given an open neighbourhood $U$ of $e\in G$, there exists a smaller open neighbourhood $V$ of $e\in G$ such that $V^{-1}V\subset U$. By the definition of compactness, see \cite[p. 164]{Mu}, there exists a finite covering of  the compact set $K$ by left translates $u_jV$, and from this we can as above produce a partition of $K$ in finitely many non-empty disjoint Borel sets, say $T_1,\ldots,T_q$, such that each
$T_k$ is contained in a left translate of $V$, and therefore we have
\begin{equation*}\label{eq:tech}
 u^{-1}v\in V^{-1}V\subset U,\quad  u,v\in T_k, k=1,\ldots,q.
\end{equation*}

Finally the sets $M_{j,k}:=S_j\times T_k$ will form a finite partition of $\mathbb S^d\times K$ with the desired properties.
\end{proof}

\begin{lemma}\label{thm:equi} For a continuous function $f:[-1,1]\times G\to \mathbb C$ the following are equivalent:
\begin{enumerate}
\item[{\rm(i)}] $f \in \mathcal P(\S^d,G)$.

\item[{\rm(ii)}] $f$ is bounded and for any complex Radon measure $\mu$ on $\mathbb S^d\times G$ of compact support we have
\begin{equation}\label{eq:mu}
\int_{\mathbb S^d\times G}\int_{\mathbb S^d\times G} f(\cos \theta(\xi,\eta),u^{-1}v)\,{\rm d}\mu(\xi,u)\,
{\rm d}\overline{\mu(\eta,v)}\ge 0.
\end{equation} 
\end{enumerate}
\end{lemma}

\begin{proof} "(i)$\implies $(ii)." Suppose first that (i) holds. By Lemma~\ref{thm:bd} we know that $f$ is bounded, and for any discrete complex Radon measure of the form
$$
\sigma=\sum_{j=1}^n \alpha_j\delta_{(\xi_j,u_j)},
$$
where $(\xi_1,u_1),\ldots (\xi_n,u_n)\in \mathbb S^d\times G$, $\alpha_1,\ldots, \alpha_n\in\mathbb C$, we have
\begin{eqnarray*}
\lefteqn{\int_{\mathbb S^d\times G}\int_{\mathbb S^d\times G} f(\cos \theta(\xi,\eta),u^{-1}v)\,{\rm d}\sigma(\xi,u)\,{\rm d}\overline{\sigma(\eta,v)}} \\
&=& \sum_{k,l=1}^n f(\cos \theta(\xi_k,\xi_l),u_k^{-1}u_l)\alpha_k\overline{\alpha_{l}}\ge 0.
\end{eqnarray*}

Let now $\mu$ denote an arbitrary complex Radon measure on $\mathbb S^d\times G$ with compact support. We may therefore suppose that the support of $\mu$ is contained in $\mathbb  S^d\times K$ for a compact set $K\subset G$. Let $M_j,j=1,\ldots,r$ be a partition of $\mathbb S^d\times K$ corresponding to given $\delta>0$ and open $U$ as in Lemma~\ref{thm:top}.

Let us now consider the number
$$
I:=\int_{\mathbb S^d\times G}\int_{\mathbb S^d\times G} f(\cos \theta(\xi,\eta),u^{-1}v)\,{\rm d}\mu(\xi,u)\,
{\rm d}\overline{\mu(\eta,v)},
$$
which is clearly real. We shall prove that $I\ge 0$, by showing that for  any $\varepsilon >0$ there
exists $J\ge0$ such that $|I-J|<\varepsilon||\mu||^2$, where $||\mu||$ is the total variation of the complex measure $\mu$, cf. \cite{R1}.

First of all
 $f(\cos\theta(\xi,\eta),u^{-1}v)$ is uniformly continuous on the compact set $(\mathbb S^d\times K)^2$, and for given $\varepsilon>0$ we therefore know that there exists
$\delta>0$ and an open neighbourhood $U$ of $e\in G$ such that we for all quadruples
 $(\xi,u, \eta,v),(\tilde\xi,\tilde u,\tilde\eta,\tilde v)\in (\mathbb S^d\times K)^2$ satisfying
\begin{eqnarray*}
\theta(\xi,\tilde\xi)<\delta,\;\theta(\eta,\tilde\eta)<\delta,\; u^{-1}\tilde u\in U,\; v^{-1}\tilde v\in U
\end{eqnarray*}
have
$$
 |f(\cos \theta(\xi,\eta),u^{-1}v)-f(\cos \theta(\tilde\xi,\tilde\eta),\tilde u^{-1}\tilde v)|<\varepsilon,
$$
in particular
\begin{eqnarray*}
|f(\cos \theta(\xi,\eta),u^{-1}v)-f(\cos \theta(\tilde\xi,\tilde\eta),\tilde u^{- 1}\tilde v)|<\varepsilon.
\end{eqnarray*}
if 
$$
(\xi,u, \eta,v),(\tilde\xi,\tilde u,\tilde\eta,\tilde v)\in M_k\times M_l.
$$

Next
$$
I=\sum_{k,l=1}^r\int_{M_k}\int_{M_l} f(\cos \theta(\xi,\eta),u^{-1}v)\,{\rm d}\mu(\xi,u)\,
{\rm d}\overline{\mu(\eta,v)},
$$
and if we choose an arbitrary point $(\xi_j,u_j)\in M_j$ and define $\alpha_j=\mu(M_j)$, $j=1,\ldots,r$, then 
$$
J:=\sum_{k,l=1}^r f(\cos \theta(\xi_k,\xi_l),u_k^{-1}u_l)\alpha_k\overline{\alpha_l}\ge 0.
$$
Furthermore,
$$
I-J=\sum_{k,l=1}^r\int_{M_k}\int_{M_l} [f(\cos \theta(\xi,\eta),u^{-1}v)
-f(\cos \theta(\xi_k,\xi_l),u_k^{-1}u_l)]
\,{\rm d}\mu(\xi,u)\,
{\rm d}\overline{\mu(\eta,v)},
$$
and  by the uniform continuity
$$
|I-J|\le \varepsilon\sum_{k,l=1}^r |\mu|(M_k)|\mu|(M_l)= \varepsilon ||\mu||^2,
$$
where $|\mu|$ denotes the total variation measure and $||\mu||$ its total mass called the total variation of $\mu$.

\medskip  

The implication  "(ii)$\implies$(i)" is easy by specializing the  measure $\mu$ to a complex discrete
measure concentrated in finitely many points. 
\end{proof}

\begin{rem}\label{thm:rem1} {\rm In Lemma~\ref{thm:equi} it is possible to add the following equivalent condition:

(iii) $f$ is bounded and for any bounded complex Radon measure $\mu$ on $\mathbb S^d\times G$ we have
\begin{equation}\label{eq:mub}
\int_{\S^d\times G}\int_{\S^d\times G} f(\cos \theta(\xi,\eta),u^{-1}v)\,{\rm d}\mu(\xi,u)\,
{\rm d}\overline{\mu(\eta,v)}\ge 0.
\end{equation} 
The equivalence follows by approximating $\mu$ with measures of compact support. Since the result is not needed, we leave the details to the reader.}
\end{rem}

{\it Proof of Theorem~\ref{thm:main}.} Suppose that $f$ belongs to $\mathcal P(\S^d,G)$ and let us consider the product measure $\mu:=\omega_d\otimes \sigma$ on $\mathbb S^d\times G$,
where $\sigma$ is an arbitrary complex Radon measure on $G$ of compact support. By Lemma~\ref{thm:equi} applied to $\mu$ we get
\begin{equation}\label{eq:help1}
\int_{\S^{d}}\int_{\S^{d}}\int_G \int_G f(\cos \theta(\xi,\eta),u^{-1}v)\,{\rm d}\omega_{d}(\xi)\,{\rm d}\omega_{d}(\eta)\,{\rm d}\sigma(u)\,
{\rm d}\overline{\sigma(v)}\ge 0.
\end{equation}
The  integral with respect to $\xi,\eta$ can be simplified to
$$
\int_{\S^{d}}\int_{\S^{d}} f(\cos \theta(\xi,\eta),u^{-1}v)\,{\rm d}\omega_{d}(\xi)\,{\rm d}\omega_{d}(\eta)=\sigma_{d}\int_{\S^{d}}
f(\varepsilon_1\cdot\eta,u^{-1}v)\,{\rm d}\omega_d(\eta),
$$
where $\varepsilon_1$ is the unit vector $(1,0,\ldots,0) \in \R^{d+1}$.
This is because
$$
\int_{\S^d}f(\xi\cdot\eta,u^{-1}v)\,{\rm d}\omega_d(\eta)
$$
is independent of $\xi\in\S^d$ due to the rotation invariance of $\omega_d$, so we can put $\xi=\varepsilon_1$. When we next integrate this constant function with respect to $\xi$, we just multiply the function with the total mass $\sigma_d$.  

Therefore Equation \eqref{eq:help1}
amounts  to
\begin{equation}\label{eq:help2}
\int_{\S^d}\int_G \int_G f(\varepsilon_1\cdot\eta,u^{-1}v)\,{\rm d}\omega_d(\eta)\,{\rm d}\sigma(u)\,
{\rm d}\overline{\sigma(v)}\ge 0.
\end{equation}
 
We next apply Equation \eqref{eq:help2} to the function $f(x,u)c_n(d,x)$ which belongs to $\mathcal P(\S^d,G)$
by Proposition~\ref{thm:elem}.
This gives
\begin{equation}\label{eq:help3}
\int_{\S^d}\int_G \int_G f(\varepsilon_1\cdot\eta,u^{-1}v)
c_n(d,\varepsilon_1\cdot\eta)\,{\rm d}\omega_d(\eta)\,{\rm d}\sigma(u)\,
{\rm d}\overline{\sigma(v)}\ge 0.
\end{equation}
The function $\varphi_{n,d}:G\to\mathbb C$ defined by
\begin{equation}\label{eq:help4}
\varphi_{n,d}(u):=\frac{N_n(d)}{\sigma_d} \int_{\S^d} f(\varepsilon_1\cdot\eta,u)
c_n(d,\varepsilon_1\cdot\eta)\,{\rm d}\omega_d(\eta)
\end{equation}
is clearly continuous and bounded, and it is positive definite on $G$, because Equation \eqref{eq:help3} holds for all complex Radon measures $\sigma$ on $G$ with compact support. The integral \eqref{eq:help4}
can be reduced to 
\begin{equation}\label{eq:help5}
\varphi_{n,d}(u)=\frac{N_n(d)\sigma_{d-1}}{\sigma_d}\int_{-1}^1 f(x,u)c_n(d,x)(1-x^2)^{d/2-1}\,{\rm d}x,
\end{equation}
cf. \cite[p.1]{M}.

For each $u\in G$ the function $\xi\mapsto f(\varepsilon_1\cdot\xi,u)$ belongs to ${\cal C}(\mathbb S^d)$ and has by \eqref{eq:exp} and \eqref{eq:proj}  an expansion in spherical harmonics

$$
f(\varepsilon_1\cdot\xi,u)=\sum_{n=0}^\infty S_n(\xi,u) 
$$
known to be Abel-summable to the  left-hand side for each $u\in G$, cf. \cite[Theorem 9]{M},
and
$$
S_n(\xi,u)=\frac{N_n(d)}{\sigma_d}\int_{\S^d}c_n(d,\xi\cdot\eta)f(\varepsilon_1\cdot\eta,u)\,{\rm d}\omega_d(\eta).
$$
By the Funk-Hecke formula \cite[p. 20]{M} and Equation \eqref{eq:help5} we get
$$
S_n(\xi,u)=\varphi_{n,d}(u)c_n(d,\varepsilon_1\cdot\xi),
$$
hence
\begin{equation*}\label{eq:help6}
f(\varepsilon_1\cdot\xi,u)=\sum_{n=0}^\infty \varphi_{n,d}(u)c_n(d,\varepsilon_1\cdot\xi),
\end{equation*}
and the series is Abel-summable to the left-hand side for each $u\in G$ and each $\xi\in\mathbb S^d$. For $u=e$ and $\xi=\varepsilon_1$ we have in particular
$$
\lim_{r\to 1^-}\sum_{n=0}^\infty r^n\varphi_{n,d}(e)=f(1,e),
$$  
but since $\varphi_{n,d}(e)\ge 0$, the limit above is equal to
 $\sum_{n=0}^\infty \varphi_{n,d}(e)$ by Lebesgue's monotonicity Theorem, hence
$$
\sum_{n=0}^\infty \varphi_{n,d}(e)=f(1,e)<\infty.
$$
Since
$$
|\varphi_{n,d}(u)c_n(d,x)|\le \varphi_{n,d}(e),
$$
the M-test of Weierstrass shows that the series 
\begin{equation}\label{eq:help7}
\sum_{n=0}^\infty \varphi_{n,d}(u) c_n(d,x),\quad (x,u)\in[-1,1]\times G,
\end{equation}
converges uniformly to a continuous function $\tilde f(x,u)$. By Abel's Theorem the series in \eqref{eq:help7} is also Abel summable with sum $\tilde f(x,u)$ and therefore
 $f(x,u)=\tilde f(x,u)$ for all $x\in[-1,1],u\in G$, and 
Theorem~\ref{thm:main} is proved. $\quad\square$

\medskip

{\it Proof of Proposition~\ref{thm:Z1}:} Although the proof is very similar to the proof given in \cite{Z}, we shall indicate the main steps, where it is crucial that the $d$-Schoenberg functions of   $f\in \mathcal P(\S^d,G)$ satisfy $|\varphi_{n,d}(u)|\le \varphi_{n,d}(e)$.

First of all  we note two formulas for the polynomials $c_n(d,x)$ which are special cases of formulas for general Gegenbauer polynomials
 $C^{(\lambda)}_n(x)$:

\begin{equation}\label{eq:Geg1}
d c_n'(d,x)=n(n+d-1)c_{n-1}(d+2,x),
\end{equation} 

\begin{equation}\label{eq:Geg2}
(1-x^2)c_n(d+2,x)=\frac{d}{2n+d+1}\left(c_n(d,x)-c_{n+2}(d,x)\right).
\end{equation}

These formulas lead to  
\begin{eqnarray*}\lefteqn{
(1-x^2)\sum_{n=1}^N \varphi_{n,d}(u)c'_n(d,x)}\\
&=&\sum_{n=1}^N\varphi_{n,d}(u)\frac{n(n+d-1)}{2n+d-1}\left(c_{n-1}(d,x)-c_{n+1}(d,x)\right)\\
&=& \sum_{n=0}^N\left(\varphi_{n+2,d}(u)\frac{(n+2)(n+d+1)}{2n+d+3}-\varphi_{n,d}(u)\frac{n(n+d-1)}{2n+d-1}\right)c_{n+1}(d,x)\\
&+& \varphi_{1,d}(u)\frac{d}{d+1}-\sum_{n=N-1}^N \varphi_{n+2,d}(u)\frac{(n+2)(n+d+1)}{2n+d+3}c_{n+1}(d,x)\\
&=& \sum_{n=0}^N \left(\frac{d(2n+d+1)(n+2)}{(2n+d+3)(n+d)}\varphi_{n+2,d}(u)-\frac{dn}{n+d}\varphi_{n,d+2}(u)\right)c_{n+1}(d,x)\\
 &+& \varphi_{1,d}(u)\frac{d}{d+1}-\sum_{n=N-1}^N \varphi_{n+2,d}(u)\frac{(n+2)(n+d+1)}{2n+d+3}c_{n+1}(d,x).
\end{eqnarray*}
In the last equality we have used Equation \eqref{eq:Schcoef3}, and we have for simplicity assumed $d\ge 2$, otherwise we shall use the other equations of Proposition~\ref{thm:Schfunctions}.
 The point is now that the coefficients in front of $\varphi_{n+2,d}(u)$  and $\varphi_{n,d+2}(u)$ form a  bounded sequence of nonnegative numbers, and since $\sum \varphi_{n+2,d}(e)<\infty$, $\sum\varphi_{n,d+2}(e)<\infty$, we can let $N\to\infty$. The  second trick of Ziegel is to notice that the two terms in  $\sum_{n=N-1}^N$ tend to $0$ for $N\to\infty$. To see this it is enough to prove that $n\varphi_{n,d}(e)\to 0$, and this is done like in \cite{Z}.

\medskip
{\it Proof of Theorem~\ref{thm:main2}:}

As preparation for the proof we need the following lemma.

\begin{lemma}[Lemma 1 in \cite{S}]\label{thm:Schtech} Let $c_n^{(\lambda)}(x)=C_n^{(\lambda)}(x)/C_n^{(\lambda)}(1)$ denote the normalized Gegenbauer polynomial. For each $x$ such that $-1<x<1$ and $\varepsilon>0$ there exists $L=L(x,\varepsilon)>0$  such that
\begin{equation*}\label{eq:Sch1}
|c_n^{(\lambda)}(x)-x^n|<\varepsilon, \quad n=0,1,\ldots,
\end{equation*}
provided $\lambda>L$.
\end{lemma}

 For each $d=1,2,\ldots$ we have the series expansion
$$
f(x,u)=\sum_{n=0}^\infty \varphi_{n,d}(u) c_n(d,x),\quad x\in[-1,1],u\in G,
$$
and we know that $|\varphi_{n,d}(u)|\le \varphi_{n,d}(e)\le f(1,e)$.

For a topological space $X$ and an index set $J$ we denote as in \cite[p. 113]{Mu} by $X^J$ the
product set of $J$-tuples $(x_j)_{j\in J}$, or equivalently as the set of functions $x:J\to X$, where we write $x(j)=x_j$ for $j\in J$. We equip $X^J$ with the product topology. 

 Let now  $X=\{z\in\C\,|\,|z|\le f(1,e)\}$ which is a compact disc in $\C$, and consider the index set $J=\N_0\times G$, where $\N_0=\{0,1,\ldots\}$. The  product set
$X^{\N_0\times G}$  equipped with the product topology is compact by Tychonoff's Theorem, and this implies that the sequence
\begin{equation}\label{eq:Sch11}
d\to (\varphi_{n,d}(u))_{(n,u)\in \N_0\times G}
\end{equation}  
 has an accumulation point  in $X^{\N_0\times G}$, which can be considered as
a sequence of functions $\varphi_n:G\to X, n\in\N_0$.  Being an accumulation point means that for each $\varepsilon>0$ and each finite set $F\subset \N_0\times G$ we have
\begin{equation}\label{eq:Sch12}
|\varphi_{n,d}(u)-\varphi_n(u)|<\varepsilon,\;  (n,u)\in F
\end{equation}
for infinitely many values of $d$. Therefore each $\varphi_n$ is a positive definite function on $G$. We stress that it is not clear at all at this point that they are continuous.

We now fix $x$ such that $-1<x<1$ and $u_0\in G$, and  shall prove that
\begin{equation}\label{eq:Sch21}
f(x,u_0)=\sum_{n=0}^\infty \varphi_n(u_0)x^n.
\end{equation} 
We first note that the right-hand side converges because $|\varphi_n(u_0)|\le f(1,e)$, and similarly
$\sum \varphi_{n,d}(u_0)x^n$ converges for each $d$.

Let now $\varepsilon>0$ be given and let $L=L(x,\varepsilon)$ be the number from Lemma~\ref{thm:Schtech}.
We next choose  $N\in\N$ such that
$$
\sum_{n=N+1}^\infty |x|^n < \varepsilon.
$$
From Equation \eqref{eq:Sch12} (with $F=\{0,1,\ldots,N\}\times\{u_0\}$) follows that 
 we can choose $d\in\N$ such that $(d-1)/2>L$ and such that
$$
|\varphi_{n,d}(u_0)-\varphi_n(u_0)|<\varepsilon,\; n=0,1,\ldots,N.
$$
 We then have
$$
f(x,u_0)-\sum_{n=0}^\infty \varphi_n(u_0)x^n=
 \sum_{n=0}^\infty \varphi_{n,d}(u_0)(c_n(d,x)-x^n) +\sum_{n=0}^\infty (\varphi_{n,d}(u_0)-\varphi_n(u_0))x^n,
$$
hence
\begin{eqnarray*}
\lefteqn{|f(x,u_0)-\sum_{n=0}^\infty \varphi_n(u_0)x^n|\le}\\
&& \sum_{n=0}^\infty \varphi_{n,d}(e)|c_n(d,x)-x^n|+\sum_{n=0}^N |\varphi_{n,d}(u_0)-\varphi_n(u_0)||x|^n + \sum_{n=N+1}^\infty 2f(1,e)|x|^n\\
&&<  \varepsilon f(1,e) + \frac{\varepsilon}{1-|x|}+2\varepsilon f(1,e),
\end{eqnarray*}
where we have used Lemma~\ref{thm:Schtech}.

Since the right-hand side can be made arbitrarily small, we have proved \eqref{eq:Sch21}.
 
Putting $u_0=e$ in Equation \eqref{eq:Sch21}, we have in particular
$$
f(x,e)=\sum_{n=0}^\infty \varphi_n(e)x^n,\quad -1<x<1,
$$
and since $\varphi_n(e)\ge 0$ we get by Lebesgue's Monotonicity Theorem that
$$
f(1,e)=\lim_{x\to 1^{-}} f(x,e)=\sum_{n=0}^\infty \varphi_n(e).
$$
In particular $\sum_{n=0}^\infty \varphi_n(e)<\infty$. Combined with the inequality 
$|\varphi_n(u)|\le \varphi_n(e),\;u\in G$ for the positive definite function $\varphi_n$, this shows that the right-hand side of \eqref{eq:Sch21} is continuous for $x\in[-1,1]$, and finally
\eqref{eq:Sch21} holds for all $(x,u)\in[-1,1]\times G$.

By Corollary~\ref{thm:Z2} we know that $\partial^n f(x,u)/\partial x^n$ is  continuous on $]-1,1[\times G$ for all $n$ and in particular
\begin{equation}\label{eq:Sch22}
\frac{\frac{\partial^n f(0,u)}{\partial x^n}}{n!}=\varphi_n(u)
 \end{equation}
is continuous.

We finally have to prove Equation \eqref{eq:expandpp}. We have shown the identity 
$$
f(x,u)=\sum_{n=0}^\infty \varphi_n(u)x^n,\quad (x,u)\in [-1,1]\times G
$$
for any accumulation point $(\varphi_n(u))$ in $X^{\N_0\times G}$ of the sequence 
$$
d\to (\varphi_{n,d}(u))_{(n,u)\in \N_0\times G}.
$$
However, because of \eqref{eq:Sch22}, the accumulation point is uniquely determined, and then it is easy to see that the sequence \eqref{eq:Sch11} converges to the unique accumulation point in the compact space $X^{\N_0\times G}$. In fact, if the convergence does not take place, there exists an open neighbourhood $U$ of $(\varphi_n(u))$ in  $X^{\N_0\times G}$ and a subsequence 
$$
(\varphi_{n,d_k}(u)),\; d_1<d_2<\ldots
$$
such that $(\varphi_{n,d_k}(u))\notin U$ for all $k\in\N$. By compactness of
 $X^{\N_0\times G}\setminus U$ this subsequence has an accumulation point in
$X^{\N_0\times G}\setminus U$, but this accumulation point is also an accumulation point of the sequence \eqref{eq:Sch11} and hence equal to $(\varphi_n(u))$, which is a contradiction.
 $\quad\square$

\section{$d$-Schoenberg functions for $f\in\mathcal P(\S^\infty,G)$}

In this section we shall describe the $d$-Schoenberg functions for $f\in\mathcal P(\S^\infty,G)$.
Although this is quite elementary we have included it because of the potential applications in geostatistics. We need the connection coefficients between  the monomials and the normalized Gegenbauer polynomials $c_n(d,x)$. These are given e.g. in Bingham \cite[Lemma 1]{Bi}. We recall Gegenbauer's formula for the connection coefficients between two different Gegenbauer polynomials:
\begin{lemma}\label{thm:con1} For $\lambda,\mu>0$ we have
$$
C_n^{(\lambda)}(x)=\sum_{k=0}^{[n/2]} \frac{(\lambda)_{n-k}(\lambda-\mu)_k(n+\mu-2k)}{(\mu)_{n-k+1}k!}C_{n-2k}^{(\mu)}(x).
$$
\end{lemma}
A simple proof can be found in \cite[p.360]{A:A:R}. It is important to notice that the connection coefficients are non-negative when $\lambda>\mu$.

\begin{lemma}\label{thm:con2} For $\mu>0$ we have
\begin{equation}\label{eq:conc1}
x^n=\sum_{k=0}^{[n/2]}\frac{n!(n+\mu-2k)}{2^n(\mu)_{n-k+1} k!}C_{n-2k}^{(\mu)}(x),
\end{equation}
and for $\mu=0$ we have
\begin{equation}\label{eq:conc2}
x^n=\sum_{k=0}^{[n/2]} \frac{\binom{n}{k}}{2^n} N_{n-2k}(1)T_{n-2k}(x),
\end{equation}
where $N_n(1)=2$ for $n\ge 1$ and $N_0(1)=1$, cf. \eqref{eq:dim}.
\end{lemma}

\begin{proof} We refer to \cite[Lemma 1]{Bi}, but for the convenience of the reader we indicate a proof. Divide the formula in Lemma~\ref{thm:con1} by $(2\lambda)_n/n!$ in order to get the normalized Gegenbauer polynomials $c_n^{(\lambda)}(x)$, and next  use that the $n$'th normalized Gegenbauer polynomial tends to $x^n$ for $\lambda\to\infty$ when $-1<x<1$, cf. Lemma~\ref{thm:Schtech}. This shows formula \eqref{eq:conc1} for these values of $x$, but as an identity between polynomials, it then holds for all values of $x$. Formula \eqref{eq:conc2} can be found e.g. in \cite{G:R}. 
\end{proof} 

It is important to notice that the connection coefficients of Lemma~\ref{thm:con2} are alternating strictly positive and 0. 

Specializing to $\mu=(d-1)/2,d\in\N$ we can simplify Lemma~\ref{thm:con2} to the following:

\begin{cor}\label{thm:con3} For $d\in\N$ and $n\in\N_0$ we have
$$
x^n=\sum_{k=0}^{[n/2]}\gamma^{(d)}(n,k)c_{n-2k}(d,x),
$$
where
\begin{equation}\label{eq:conc3}
\gamma^{(d)}(n,k)=\left\{\begin{array}{ll}
\frac{n!(d-1)_{n-2k}(n-2k+(d-1)/2)}{2^n k!(n-2k)!((d-1)/2)_{n-k+1}} & \mbox{if $d\ge 2$}\\
\frac{\binom{n}{k}}{2^n}N_{n-2k}(1) & \mbox{if $d=1$}.
\end{array}
\right.
\end{equation} 
\end{cor}

Using the notation above, we can now give the main result of this section.

\begin{thm}\label{thm:infty-d} For $f\in\mathcal P(\S^\infty,G)$ with  the representation
\begin{equation}\label{eq:e1}
f(x,u)=\sum_{n=0}^\infty \varphi_n(u)x^n,\quad (x,u)\in [-1,1]\times G,
\end{equation}
the $d$-Schoenberg functions in the representation of $f\in\mathcal P(\S^d,G)$ as
$$
f(x,u)=\sum_{n=0}^\infty \varphi_{n,d}(u)c_n(d,x),\quad (x,u)\in [-1,1]\times G,
$$
are given as 
\begin{equation}\label{eq:conc4}
\varphi_{n,d}(u)=\sum_{j=0}^\infty \varphi_{n+2j}(u)\gamma^{(d)}(n+2j,j),\quad u\in G
\end{equation}
and the series is uniformly convergent for $u\in G$. Here
\begin{equation}\label{eq:conc5}
\gamma^{(d)}(n+2j,j)=\left\{\begin{array}{ll}
\frac{(n+2j)!(d-1)_{n}(n+(d-1)/2)}{n! 2^{n+2j} j! ((d-1)/2)_{n+j+1}} & \mbox{if $d\ge 2$}\\
\frac{\binom{n+2j}{j}}{2^{n+2j}}N_{n}(1) & \mbox{if $d=1$}.
\end{array}
\right.
\end{equation}
\end{thm} 

\begin{proof} Inserting the uniformly convergent expansion \eqref{eq:e1} in Equation \eqref{eq:coef} for $\varphi_{n,d}(u)$, we get after interchanging integration and summation 
\begin{eqnarray*}
\lefteqn{\varphi_{n,d}(u)}\\
&=&\sum_{p=0}^\infty \varphi_p(u)\frac{N_n(d)\sigma_{d-1}}{\sigma_d}\int_{-1}^1 x^p c_n(d,x)(1-x^2)^{d/2-1}\,{\rm d}x\\
&=&\sum_{p=0}^\infty \varphi_p(u) \sum_{j=0}^{[p/2]}\gamma^{(d)}(p,j) \frac{N_n(d)\sigma_{d-1}}{\sigma_d}\int_{-1}^1 c_{p-2j}(d,x) c_n(d,x)(1-x^2)^{d/2-1}\,{\rm d}x,
\end{eqnarray*}
where we have used Corollary~\ref{thm:con3}. 

The last integral is zero unless $p-2j=n$, which requires $p=n+2j,\,j\in\N_0$. 
Using \eqref{eq:orthspec} we get
$$
\varphi_{n,d}(u)=\sum_{j=0}^\infty \varphi_{n+2j}(u)\gamma^{(d)}(n+2j,j),
$$ 
and the expression in \eqref{eq:conc5} is obvious from Corollary~\ref{thm:con3}.  By Stirling's formula, the expression
$\gamma^{(d)}(n+2j,j)$  is easily seen to be bounded for $j\in\N_0$, and therefore the 
series in Equation \eqref{eq:conc4} is uniformly convergent.
\end{proof}

The special case of Equation \eqref{eq:conc4} where $G=\{e\}$ is given in \cite[Theorem 4.2 (b)]{M:N:P:R}.

\section{Applications to some homogeneous spaces}

In a recent manuscript \cite{G:M:P} the authors prove characterization results for isotropic positive definite kernels on $\S^d\times \S^{d'}$ for $d,d'\in\N\cup\{\infty\}$. They consider the set
$\mathcal P(\S^d,\S^{d'})$  of continuous functions $f:[-1,1]^2\to\R$ such that the kernel $K$ of Equation \eqref{eq:gmp} is  positive definite.

Specializing Theorem~\ref{thm:main} and Theorem~\ref{thm:main2} to the compact group
$G=O(d'+1),d'\in\N$ of orthogonal transformations of $\R^{d'+1}$, we can deduce Theorem 2.9 and Theorem 3.4 from \cite{G:M:P}. Their Theorem 3.5 corresponding to the case where $d=d'=\infty$ cannot be deduced  from our result, since the group of unitary operators in an infinite dimensional Hilbert space is not compact.   

\begin{thm}[Theorem 2.9 of \cite{G:M:P}] Let $d,d'\in\N$ and let $f:[-1,1]^2\to\R$ be a continuous function. Then $f\in\mathcal P(\S^d,\S^{d'})$ if and only if
\begin{equation}\label{eq:2pd}
f(x,y)=\sum_{n,m=0}^\infty \widehat{f}_{n,m}c_n(d,x)c_m(d',y),\quad x,y\in[-1,1],
\end{equation} 
where $\widehat{f}_{n,m}\ge 0$ such that $\sum \widehat{f}_{n,m}<\infty$.

The  above expansion is uniformly convergent, and we have
\begin{eqnarray}\label{eq:coef2pd}\lefteqn{
\widehat{f}_{n,m}=\frac{N_n(d)\sigma_{d-1}}{\sigma_d}\frac{N_m(d')\sigma_{d'-1}}{\sigma_{d'}}}\nonumber\\
&\times&\int_{-1}^1\int_{-1}^1 f(x,y)c_n(d,x)c_m(d',y)(1-x^2)^{d/2-1}(1-y^2)^{d'/2-1}\,{\rm d}x\,{\rm d}y.
\end{eqnarray}
\end{thm}

\begin{proof} It is elementary to see that the right-hand side of \eqref{eq:2pd} defines a function $f\in\mathcal P(\S^d,\S^{d'})$. 

Let us next consider  $f\in\mathcal P(\S^d,\S^{d'})$ and define $F:[-1,1]\times O(d'+1)\to\R$ by
\begin{equation*}\label{eq:biinv}
F(x,A)=f(x,A\varepsilon_1\cdot\varepsilon_1), \quad x\in [-1,1],\;A\in O(d'+1),
\end{equation*} 
where $\varepsilon_1=(1,0,\ldots,0)\in\S^{d'}$. Then $F\in\mathcal P(\S^d,O(d'+1))$ because
$$
F(x,B^{-1}A)=f(x,A\varepsilon_1\cdot B\varepsilon_1),\quad A,B\in O(d'+1).
$$

By Theorem~\ref{thm:main}
\begin{equation}\label{eq:2pd1}
F(x,A)=\sum_{n=0}^\infty \varphi_{n,d}(A)c_n(d,x),\quad x\in[-1,1],\; A\in O(d'+1),
\end{equation}
and
\begin{equation}\label{eq:2pd2}
\varphi_{n,d}(A)=\frac{N_n(d)\sigma_{d-1}}{\sigma_d}\int_{-1}^1 f(x,A\varepsilon_1\cdot\varepsilon_1)c_n(d,x)(1-x^2)^{d/2-1}\,{\rm d}x
\end{equation}
belongs to  $\mathcal P(O(d'+1))$. 

The fixed point group of $\varepsilon_1$
$$
\{A\in O(d'+1)\mid A\varepsilon_1=\varepsilon_1\},
$$
is isomorphic to $O(d')$, so we denote it $O(d')$.

The function $\varphi_{n,d}$ is bi-invariant under $O(d')$,
i.e.,
$$
\varphi_{n,d}(KAL)=\varphi_{n,d}(A),\quad A\in O(d'+1),\,K,L\in O(d').
$$
This is simply because $f(x,KAL\varepsilon_1\cdot\varepsilon_1)=f(x,A\varepsilon_1\cdot\varepsilon_1)$.

The mapping $A\mapsto A\varepsilon_1$ is a continuous surjection of $O(d'+1)$ onto $\S^{d'}$, and it induces a homeomorphism of the homogeneous space $O(d'+1)/O(d')$ onto $\S^{d'}$.

It is easy to see that as a bi-invariant function, $\varphi_{n,d}$ has the form 
$$
\varphi_{n,d}(A)=g_{n,d}(A\varepsilon_1\cdot\varepsilon_1)
$$
for a uniquely determined continuous function $g_{n,d}:[-1,1]\to\R$. We have in addition $g_{n,d}\in\mathcal P(\S^{d'})$, because for $\xi_1,\ldots,\xi_n\in\S^{d'}$ there exist 
$A_1,\ldots,A_n\in O(d'+1)$ such that $\xi_j=A_j\varepsilon_1,j=1,\ldots,n$, hence
$$
g_{n,d}(\xi_k\cdot\xi_l)=g_{n,d}(A_l^{-1}A_k\varepsilon_1\cdot\varepsilon_1)=\varphi_{n,d}(A_l^{-1}A_k).
$$
This means that Equation \eqref{eq:2pd2} can be written
\begin{equation*}\label{eq:2pd2a}
g_{n,d}(y)=\frac{N_n(d)\sigma_{d-1}}{\sigma_d}\int_{-1}^1 f(x,y)c_n(d,x)(1-x^2)^{d/2-1}\,{\rm d}x,\quad y\in[-1,1].
\end{equation*}
By Schoenberg's Theorem~\ref{thm:S-sphere}
\begin{equation}\label{eq:2pd3} 
g_{n,d}(y)=\sum_{m=0}^\infty b^{(n,d)}_{m,d'}c_m(d',y),\quad y\in [-1,1],
\end{equation}
where
$$
b^{(n,d)}_{m,d'}=\frac{N_m(d')\sigma_{d'-1}}{\sigma_{d'}}\int_{-1}^1g_{n,d}(y)c_m(d',y)(1-y^2)^{d'/2-1}\,{\rm d}y,
$$
hence by Equations \eqref{eq:2pd1}, \eqref{eq:2pd3}
$$
F(x,A)=f(x,A\varepsilon_1\cdot\varepsilon_1)=\sum_{n=0}^\infty\left(\sum_{m=0}^\infty b^{(n,d)}_{m,d'}c_m(d',A\varepsilon_1\cdot\varepsilon_1)\right)c_n(d,x),
$$
which is equivalent to Equation \eqref{eq:2pd}, and $\widehat{f}_{n,m}=b^{(n,d)}_{m,d'}$ is given by Equation \eqref{eq:coef2pd}.
\end{proof}

\begin{thm}[Theorem 3.4 of \cite{G:M:P}] Let $d'\in\N$ and let $f:[-1,1]^2\to\R$ be a continuous function. Then $f\in\mathcal P(\S^{\infty},\S^{d'})$ if and only if
\begin{equation}\label{eq:2pd4}
f(x,y)=\sum_{n,m=0}^\infty \widetilde{f}_{n,m}x^n c_m(d',y),\quad x,y\in[-1,1],
\end{equation} 
where $\widetilde{f}_{n,m}\ge 0$ such that $\sum \widetilde{f}_{n,m}<\infty$.

The  above expansion is uniformly convergent, and we have
$$
\widetilde{f}_{n,m}=\frac{N_m(d')\sigma_{d'-1}}{\sigma_{d'} n!}\int_{-1}^1\frac{\partial^n f(x,y)}{\partial x^n}|_{x=0}\,c_m(d',y)(1-y^2)^{d'/2-1}\,{\rm d}y.
$$
\end{thm}

The proof is analogous to the proof of the previous theorem.

\subsection*{Acknowledgments}
{\small This work was initiated during the visit of the first author to Universidad T\'ecnica Federico Santa Maria, Chile. The visit and E.P. have been supported by Proyecto Fondecyt Regular.

The authors want to thank two independent referees for useful suggestions and references.}

\noindent
Christian Berg\\
Department of Mathematical Sciences, University of Copenhagen\\
Universitetsparken 5, DK-2100, Denmark\\
e-mail: {\tt{berg@math.ku.dk}}

\vspace{0.4cm}
\noindent
Emilio Porcu \\
Department of Mathematics, Universidad T{\'e}cnica Federico Santa Maria\\
Avenida Espa{\~n}a 1680, Valpara{\'\i}so, 2390123, Chile\\
e-mail: {\tt{emilio.porcu@usm.cl}}


\begin{thebibliography}{xxx}
\bibitem{A:A:R} G.~E.~Andrews, R.~Askey and R.~Roy, Special Functions. Cambridge University Press 1999. 

\bibitem{Ba} C.~Bachoc, Semidefinite programming, harmonic analysis and coding theory. ArXiv:0909.4767v2, 2010.

\bibitem{B:M} V.~S.~Barbosa and V.~A.~Menegatto, Differentiable positive definite functions on two-point homogeneous spaces. J. Math. Anal. Appl. {\bf 434} (2016), 698--712.

\bibitem{B} C.~Berg, Corps convexes et potentiels sph{\'e}riques, Mat.--Fys. Medd. Danske Vid. Selsk. {\bf 37} No. 6 (1969), 64 pages. 

\bibitem{berg2008} C.~Berg, Stieltjes-Pick-Bernstein-Schoenberg and their
connection to complete monotonicity. Pages 15--45 in  Positive Definite Functions: From Schoenberg to Space-Time challenges. J.~Mateu and E.~Porcu eds. Castell{\'o}n de la Plana 2008.

\bibitem{Bh} R.~Bhatia, Positive Definite Matrices. Princeton University Press, Princeton, 
New Jersey, 2007.

\bibitem{Bi} N.~H.~Bingham, Positive definite functions on spheres. Mathematical Proceedings of the Cambridge Philosophical Society {\bf 73} (1973), 145--156. 

\bibitem{Bo} S.~Bochner, Hilbert distances and positive definite functions. Annals of Mathematics
{\bf 42}, 3 (1941), 647--656.

\bibitem{djdep2013} D.~J.~Daley and E.~Porcu, Dimension walks and Schoenberg spectral measures. Proceedings of the American Mathematical Society, {\bf 142}, 5 (2014), 1813--1824.

\bibitem{D:X} F.~Dai and Y.~Xu, Approximation Theory and Harmonic Analysis on Spheres and Balls. Springer monographs in mathematics. Springer, New York 2013.

\bibitem{D} J.~Dixmier, Les $C^*$-alg{\`e}bres et leurs repr{\'e}sentations. Gauthier-Villars, Paris, 1964.

\bibitem{gneiting2013} T.~Gneiting, Strictly and non-strictly positive definite functions on spheres,
 Bernoulli {\bf 19}(4),  (2013), 1327--1349.

\bibitem{Gn} T.~Gneiting,  Online supplemental materials to "Strictly and non-strictly positive definite functions on spheres"
doi:10.3150/12-BEJSP06. http://projecteuclid.org/euclid.bj/1377612854.

\bibitem{G:R} I.~S.~Gradshteyn and I.~M.~Ryzhik,  Table of
  Integrals, Series and Products, Sixth Edition. Academic Press, San
Diego, 2000.

 \bibitem{G:M:P} J.~C.~Guella, V.~A.~Menegatto and A.~P.~Peron, An extension of a theorem of Schoenberg to products of spheres. arXiv:1503.08174.

\bibitem{H:J} R.~A.~Horn and C.~R.~Johnson, Matrix Analysis. Corrected reprint of the 1985 original. Cambridge University Press, Cambridge, 1990.

\bibitem{Mu} J.~R.~Munkres, Topology, Second Edition. Prentice Hall, Inc. New Jersey, 2000.

\bibitem{M} C.~M\"uller, Spherical Harmonics. Lecture Notes in Mathematics. Springer-Verlag.
Berlin, Heidelberg, New York, 1966.

\bibitem{M:N:P:R} J.~M{\o}ller, M.~Nielsen, E.~Porcu and E.~Rubak,
Determinantal point process models on the sphere. Submitted.

\bibitem{P:B:G} E.~Porcu, M.~Bevilacqua and M.~G.~Genton, Spatio-Temporal Covariance and Cross-Covariance Functions of the Great Circle Distance on a Sphere, Journal of the American Statistical Association. Forthcoming.


\bibitem{R} W.~Rudin, Fourier Analysis on Groups. Interscience Publishers, New York, London, 1962.

\bibitem{R1} W.~Rudin, Real and Complex Analysis. McGraw-Hill Book Company, Singapore,  1986. 

\bibitem{sasvari} Z.~Sasv{\'a}ri, Positive Definite and Definitizable Functions.  Akademie Verlag, Berlin, 1994.

\bibitem{S} I.~J.~Schoenberg, Positive definite functions on spheres,  Duke Math. J.  {\bf 9} (1942), 96--108.

\bibitem{Sh} V.~L.~Shapiro, Fourier series in several variables with applications to partial differential equations. Chapman \& Hall/CRC Applied Mathematics and Nonlinear Science Series. CRC Press, Boca Raton, FL, 2011.

\bibitem{Z} J.~Ziegel, Convolution roots and differentiability of isotropic positive definite functions on spheres, Proc. Amer. Math. Soc. {\bf 142}, No. 6 (2014), 2063--2077.
\end{thebibliography}
\end{document}